\documentclass[a4paper]{article}

\usepackage[utf8]{inputenc}
\usepackage[T1]{fontenc}
\usepackage[francais,english]{babel}

\usepackage{mathtools,amsmath,amsthm,amssymb}
\usepackage{tikz}
\usetikzlibrary{decorations.pathreplacing}
\usetikzlibrary{arrows}
\usepackage{stmaryrd}
\usepackage{array,makecell}
\usepackage{caption}
\usepackage{framed}
\usepackage{fullpage}

\renewcommand{\leq}{\leqslant}
\renewcommand{\geq}{\geqslant}

\newcommand{\R}{{\mathbb{R}}}
\newcommand{\e}{\varepsilon}

\newcommand{\PP}{\mathbb{P}}
\newcommand{\EE}{{\mathbb E}}

\newcommand{\supp}{{\rm supp}}
\newcommand{\dd}{{\rm d}}

\theoremstyle{definition}
\newtheorem{mydef}{Definition}[section]

\newtheorem{proposition}[mydef]{Proposition}
\newtheorem{theorem}[mydef]{Theorem}

\newtheorem{assumption}[mydef]{Assumption}
\newtheorem{lemma}[mydef]{Lemma}

\title{Averaging for some simple constrained Markov processes.}

\author{Alexandre Genadot}              

\date{Institut de Math\'ematiques de Bordeaux and Inria Bordeaux -- Sud Ouest, Team CQFD}

\begin{document}
\maketitle
\begin{abstract}
In this paper, a class of  piecewise deterministic Markov processes with underlying fast dynamic is studied. Using a ``penalty method'', an averaging result is obtained when the underlying dynamic is infinitely accelerated. The features of the averaged process, which is still a piecewise deterministic Markov process, are fully described. 
\end{abstract}

\section{Introduction}

This paper studies some simple constrained Markov processes through averaging. Their trajectories consist in a piecewise linear motion, whose slopes are positives and given by the values of a continuous time Markov chain with countable state space. The piecewise linear process is constrained to stay above some boundary by instantaneous downward jumps when hitting the boundary. This describes a very particular class of piecewise deterministic Markov processes in the sense of \cite{Davis93}. We are interested in the limit behavior of the process when the dynamic of the underlying celerity process, that is the dynamic of the underlying continuous time Markov chain, is infinitely accelerated. We are thus in the framework of averaging for Markov processes.
\medbreak
Averaging for unconstrained Markov process, that is without the presence of a boundary, has been studied by several authors since decades and is well understood for a rich variety of Markov processes, see for example \cite{Kurtz92,PS08,YZ06} and references therein. As far as we know, averaging for constrained Markov processes, that is with the presence of a boundary, is not as well understood, in particular in the description of the averaging measure at the boundary. However, in \cite{Kurtz90}, the author proposes a general method for the study of general constrained Markov processes, the so called ``patchwork martingale problem''. For example, this method has been applied recently in \cite{CK15} to reflected diffusions. We adopt in this paper a more standard approach, at least in our point of view, which is the ``penalty method'', exposed in \cite[Section 6.4]{Kurtz90}. This method consist in considering a penalized process jumping at fast rate when beyond the boundary rather than  a process jumping instantaneously at the boundary. Then, a time change is performed in order to sufficiently slow down the dynamic of the penalized process when beyond the boundary, allowing the application of classical limit theorems for Markov processes.\medbreak
In Section \ref{sec:plmp}, piecewise linear Markov processes are presented. Our main averaging result is stated in Section \ref{sec:av:plmp}. In our case, the averaged process can be fully described. In particular, the expression for the averaging measure at the boundary, describing the behavior of the limit process at the boundary, is explicitly given in terms of the features of the process. By an appropriate change of variable, this allows us, in Section \ref{sec:av:gen}, to apply this averaging result to a more general class of piecewise deterministic process than piecewise linear. As an example, a hybrid version of a classical model for the neural dynamic is considered in Section \ref{sec:quad}. Proofs are differed to Section \ref{sec:proofs}.

\section{Model and main results}

\subsection{A piecewise linear Markov process}\label{sec:plmp}

All our random variables and processes are defined on a same probability space $(\Omega, \mathcal{F},\PP)$ with associated expectation denoted by $\EE$. Convergence in law for processes is intended to take place in the Skorokhod space of càdlàg processes $\mathbb{D}([0,T],\R)$, with some finite horizon time $T$, endowed with its usual topology, see \cite[Section 12, Chapter 3]{Bil13}.

Let $c$ be a real representing some threshold or boundary. We are going to describe, at first in an algorithmic fashion, the dynamic of a stochastic process $(X(t),t\in[0,T])$ valued in $(-\infty,c)$ endowed with its Borel algebra ${\cal B}(-\infty,c)$:
\begin{enumerate}
\item{\bf Initial state:} At time $T^\ast_0=0$, the process starts at $X(T^\ast_0)=\xi_0$, a random variable with law with support included in $(-\infty,c)$.
\item{\bf First jumping time:} Let $Y$ be a continuous time Markov chain valued in a countable space ${\cal{Y}}\subset(0,\infty)$. This chain starts at $Y(0)=\zeta$, a $\cal{Y}$-valued random variable. The first hitting time of the boundary occurs at time $T^\ast_1$ defined as
$$
T^\ast_1=\inf\left\{t>0~;~\xi_0+\int_{T^\ast_0}^{t} Y(s)\dd s=c\right\}.
$$
As usually, we set $\inf\emptyset=+\infty$.
\item {\bf Piecewise linear motion:} For $t\in[T^\ast_0,T^\ast_1)$, we set
$$
X(t)=\xi_0+\int_{T^\ast_0}^{t} Y(s)\dd s.
$$
The dynamic of $X$ is thus piecewise linear here, with velocity given by $Y$.
\item {\bf Jumping measure:} At time $T^{\ast -}_1$ (the time just before $T^{\ast}_1$), the process is constrained to stay inside $(-\infty,c)$ by jumping according to the $Y$-dependent measure $\nu_{Y_{T^{\ast -}_1}}$ whose support is included in $(-\infty,c)$:
$$
\forall A\in{\cal B}(-\infty,c),\quad \PP(X(T^\ast_1)\in A)=\nu_{Y_{T^{\ast -}_1}}(A).
$$
\item {\bf And so on:} Go back to step 1 in replacing $T^\ast_0$ by $T^\ast_1$ and $\xi_0$ by $\xi_1=X(T^\ast_1)$.
\end{enumerate}

An example of a trajectory of such a process is displayed in Figure \ref{fig:ex}. The process $X$ is piecewise linear and the couple $(X,Y)$ is in fact a piecewise deterministic Markov process in the sense of \cite[Section 24, p. 57]{Davis93}. We denote by $p^\ast(t)$ the number of jumps of $X$ until time $t$:
$$
p^\ast(t)=\sum_{i=1}^{\infty}1_{T^\ast_i\leq t}.
$$
\begin{assumption}\label{n:ass}
As in \cite[Assumption 24.4, p. 60]{Davis93}, for the well definition of the process, we assume that
$$
\EE(p^\ast(T))<\infty.
$$
\end{assumption}
As stated in \cite[Theorem 31.3, p. 83]{Davis93}, the process $X$ satisfies the following martingale property, which gives another insight into the dynamic of $X$ and will be useful in the sequel. Let $f:(-\infty,c)\to \R$ such that
\begin{itemize}
\item[G1)] $f$ is measurable and absolutely continuous with respect to the Lebesgue measure;
\item[G2)] $f$ is locally integrable at the boundary: for any $t\in[0,T]$,
$$
\EE\left(\sum_{T^\ast_i\leq t}|f(X(T^\ast_i))-f(X(T^{\ast-}_i))|\right)<\infty.
$$
\end{itemize}
Then the process defined for $t\in[0,T]$ by 
\begin{align*}
&f(X(t))-f(X(0))-\int_0^t f'(X(s))Y(s)\dd s\\
&\qquad-\int_0^t \int_{-\infty}^c [f(u)-f(X(s^-))]\nu_{Y(s^-)}(\dd u) p^\ast(\dd s)
\end{align*}
is a martingale with respect to the natural filtration associated to $(X,Y)$. Notice that it is quite easy to read the piecewise linear and jump behaviors of $X$ in such a writing. Let us notice that of course, by symmetry, processes with only negative slopes may be considered in this framework.
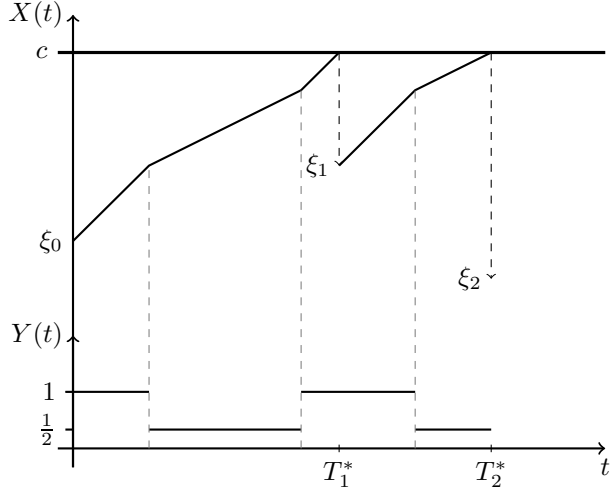
\begin{figure}
\begin{center}
\begin{tikzpicture}
\draw[thick,->](0,-1)--(0,5) node[left]{$X(t)$};
\draw[very thick](-0.2,4.5) node[left]{$c$}--(7,4.5);
\draw[thick] (0,2) node[left]{$\xi_0$}--(1,3)--(3,4)--(3.5,4.5);
\draw[dashed,->] (3.5,4.5)--(3.5,3) node[left]{$\xi_1$};
\draw[thick] (3.5,3)--(4.5,4)--(5.5,4.5);
\draw[dashed,->] (5.5,4.5)--(5.5,1.5) node[left]{$\xi_2$};
\draw[thick,->](0,-1)--(0,0.75) node[left]{$Y(t)$};
\draw[thick,->](-0.2,-0.75)--(7,-0.75) node[below]{$t$}; 
\draw[thick] (-0.1,0) --(1,0);
\draw (-0.1,0) node[left]{$1$}-- (0,0);
\draw[thick] (-0.1,-0.5) node[left]{$\frac12$}--(0,-0.5);
\draw[thick] (1,-0.5)--(3,-0.5);
\draw[thick] (3,0)--(4.5,0);
\draw[thick] (4.5,-0.5)--(5.5,-0.5);
\draw[gray,dashed] (1,-0.75)--(1,3);
\draw[gray,dashed] (3,-0.75)--(3,4);
\draw[gray,dashed] (4.5,-0.75)--(4.5,4);
\draw (3.5,-0.7)--(3.5,-0.8) node[below]{$T^\ast_1$};
\draw (5.5,-0.7)--(5.5,-0.8) node[below]{$T^\ast_2$};
\end{tikzpicture}
\caption{A trajectory of $X$ with $Y$ switching between $1/2$ and $1$.}\label{fig:ex}
\end{center}
\end{figure}

\subsection{Acceleration and averaging result}\label{sec:av:plmp}

From now on, we assume that the process of celerities, that is the continuous time Markov chain $Y$, has a fast dynamic, by introducing a (small) timescale parameter $\e$ such that
$$
\forall t\geq0,\quad Y_\e(t)=Y\left(t/\e\right).
$$
In the same time, to insure a limiting behavior, we assume that $Y$ is positive recurrent with intensity matrix $Q=(q_{zy})_{z,y\in\cal Y}$  and invariant probability measure $\pi$. For convenience, let us also define by $V$ the diagonal matrix such that $\text{diag}(V)=\{y~;~y\in{\cal Y}\}$.\\
As $\e$ goes to zero, the process $Y_\e$ converges towards the stationary state associated to $Y$ in the sense that, by the ergodic theorem,
$$
\forall t\geq0,\forall y\in{\cal Y},\quad \lim_{\e\to0}\PP(Y_\e(t)=y)=\pi(\{y\}).
$$
Therefore, as $\e$ goes to zero, the process $X_\e$, defined as $X$ by replacing $Y$ by $Y_\e$, should have its dynamic averaged with respect to the measure $\pi$. The behavior of the limiting process away from the boundary is indeed not hard to describe.
\begin{proposition}\label{prop;av:nb}
Assume that $\max\cal Y$ is finite and $\xi_0$ is deterministic. Then, for any $\eta>0$, the process $X_\e$ converges in law towards a process $\bar X$ on\\ $\mathbb{D}\left(\left[0,\frac{c-\xi_0}{\max\cal Y}-\eta\right],\R\right)$ defined as:
\begin{align*}
\bar X(t)&=\xi_0+\int_0^t\int_{\cal Y}y\pi(\dd y)\dd s\\
&=\xi_0+ \sum_{y\in\cal Y}y\pi(\{y\})t.
\end{align*}
\end{proposition}
\begin{proof}
On $\left[0,\frac{c-\xi_0}{\max\cal Y}-\eta\right]$ the process $X_\e$ does not reach the boundary. Then, classical averaging results apply (and apply to much more general situations, but still without boundary), see for example \cite{Kurtz92} and references therein.
\end{proof}

Proposition \ref{prop;av:nb} gives the behavior of the limiting process away from the boundary: celerities are averaged against the measure $\pi$. But what happens at boundary? This is what is characterized by our main result. 
Our main assumption is the following.
\begin{assumption}\label{main:ass}
The set $\cal Y$ is bounded from above and $ \EE(\sup_{\e\in(0,1]}p^\ast_\e(T))$ is finite.
\end{assumption}
Note that this ensure the well definition of the process for all $\e\in(0,1]$ since this assumption implies Assumption \ref{n:ass} for each such $\e$. Under the assumption that $\cal Y$ is bounded from above, an easy way to ensure that $\sup_{\e\in(0,1]} \EE(p^\ast_\e(T))$ is finite is to suppose that there is some $\rho>0$ such that $$\bigcup_{y\in{\cal Y}}\supp~\nu_y\subset (-\infty,c-\rho).$$ In such a case, $\sup_{\e\in(0,1]} p^\ast_\e(T)$ is even bounded by a deterministic constant ( which is $T\max{\cal Y}/\rho$). For convenience, let us write
$$
{\cal Y}^{-1}=\{y^{-1}~;~y\in {\cal Y}\}.
$$
\begin{theorem}\label{thm}
Under Assumption \ref{main:ass}, the process $X_\e$ converges in law in $\mathbb{D}([0,T],\R)$ towards a process $\bar X$ such that for any measurable function\\ $f:(-\infty,c)\to \R$ satisfying G1) and G2) the process defined by
\begin{align*}
&f(\bar X(t))-f(\xi_0)-\sum_{y\in\cal Y}y\pi(\{y\})\int_0^t f'(\bar X(s)) \dd s\\
&\qquad-\int_0^t \int_{-\infty}^c [f(u)-f(\bar X(s^-))]\bar\nu(\dd u) \bar p^\ast(\dd s)
\end{align*}
for $t\in[0,T]$, defined a martingale, with $\bar p^\ast$ the counting measure at the boundary for $\bar X$. The averaging measure at the boundary $\bar\nu$ is defined by
$$
\bar\nu(\dd u)=\sum_{y\in{\cal Y}}\nu_{y}(\dd u)\pi^\ast(\{1/y\})
$$
where $\pi^*$ is the invariant measure associated to the intensity matrix $V^{-1}Q$, thought as the generator of a $\mathcal{Y}^{-1}$-valued continuous time Markov chain.
\end{theorem}
The measure $\pi^\ast$ exists since $V^{-1}Q$ is still an irreducible transition rate matrix. Let us remark that the limiting process $\bar X$ is still a piecewise linear Markov process. Indeed, the process $\bar X$ begins at $\xi_0$ and then follows the linear motion with speed $
\sum_{y\in\cal Y}y\pi(\{y\})$ until it reaches $c$ at time
$$
T^\ast_1=\frac{c-\xi_0}{\sum_{y\in\cal Y}y\pi(\{y\})}.
$$
And so on... The next hitting times of the boundary are given recursively by
$$
T^\ast_k=T^\ast_{k-1}+\frac{c-\xi_{k-1}}{\sum_{y\in\cal Y}y\pi(\{y\})},\quad k\geq1,
$$
where the $\xi_k$'s are the post jump value locations, which are independents and distributed according to the averaging measure $\bar\nu$.
\medbreak
Let us remark that it is not surprising that the value of $Y$ actually appears through $V$ in the averaging measure at the boundary since $X$ will more likely hit the boundary when its derivative is large. Of course, this fact is compensated by the probability to be in such a high speed for $Y$; this is emphasized by the presence of the intensity matrix $Q$ in the definition of $\pi^\ast$. This indicates that in a more general setting (in greater dimension for example), the scalar product between the normal and the tangent of the flow at the boundary should be involved in the expression of $\pi^\ast$.

\subsection{Extension and application to a slow-fast hybrid quadratic integrate-and-fire models}\label{sec:av:gen}

\subsubsection{Extension and reduction to piecewise linear motions}

We can handle slightly more general motions than piecewise linear in our setting. We now consider a process $(X(t),t\in[0,T])$ which obeys to the following dynamic:
\begin{enumerate}
\item{\bf Initial state:} As before, at time $T^\ast_0=0$, the process starts at $X(T^\ast_0)=\xi_0$, a random variable with support is included in $(m,c)$ where $\{c\}$ is considered as a boundary and $m<c$ is some real.
\item{\bf First jumping time:} The first hitting time of the boundary occurs at time $T^\ast_1$ defined as
$$
T^\ast_1=\inf\left\{t>0~;~\xi_0+\int_{T^\ast_0}^{t} \alpha(Y(s))F(X(s))\dd s=c\right\},
$$
where $\alpha$ is a positive measurable function such that $\alpha({\cal Y})$ is bounded from above and $F$ is a positive continuous function.
\item {\bf Piecewise deterministic motion:} For $t\in[T^\ast_0,T^\ast_1)$, we set
$$
X(t)=\xi_0+\int_{T^\ast_0}^{t} \alpha(Y(s))F(X(s))\dd s.
$$
The dynamic of $X$ is thus continuous here, and given by the differential equation:
$$
\frac{\dd X}{\dd t}(t)=\alpha(Y(t))F(X(t)),\quad X(0)=\xi_0.
$$
\item {\bf Jumping measure:} Then, at time $T^{\ast -}_1$, the process is constrained to stay inside $(m,c)$ by jumping according to the $Y$-dependent measure $\mu_{Y_{T^{\ast -}_1}}$ whose support is included in $(m,c)$:
$$
\forall A\in{\cal B}(m,c),\quad \PP(X(T^\ast_1)\in A)=\mu_{Y_{T^{\ast -}_1}}(A).
$$
\item {\bf And so on:} Go back to step 1 in replacing $T^\ast_0$ by $T^\ast_1$ and $\xi_0$ by $\xi_1=X(T^\ast_1)$.
\end{enumerate}

With $y\in\cal Y$, the simple form of the differential equation $x'=A(y)F(x)$ allows for the following reduction. Assume that $\frac1F$ is integrable over $(m,c)$ and consider the function defined, for $x\in(m,c)$, by 
$$
G(x)=\int_{m}^x\frac{\dd u}{F(u)}.
$$
Remark that $G$ is an homeomorphism from $(m,c)$ to $(0,G(c))$. The process $Z=G(X)$ is such that for any $f:(0,G(c))\to \R$ satisfying conditions G1) and G2), the process
\begin{align*}
&f(Z(t))-f(G(\xi_0))-\int_0^t f'(Z(s))\alpha(Y(s))\dd s\\
&\qquad-\int_0^t \int_m^c [f(G(u))-f(Z(s^-))]\mu_{Y(s^-)}(\dd u) p^\ast(\dd s)
\end{align*}
is a martingale with respect to the natural filtration associated to $(Z,Y)$. It is clear from this formulation that $Z$ is a piecewise linear Markov process as in Section \ref{sec:plmp}: the process $\alpha(Y)$ is still a continuous time Markov chain with intensity matrix $Q$ but valued in $\alpha(\mathcal{Y})=\{\alpha(y)~;~y\in{\cal Y}\}$ and for $y\in\cal Y$, the jumping measure at boundary is a measure on $(0,G(c))$ given by
$$
\nu_{y}(\dd u)=\mu_y(\dd G^{-1}( u)).
$$
Note also that by construction the times where $X$ and $Z$ hit there respective boundaries $\{c\}$ and $\{G(c)\}$ are equals. The function $G$ being a homeomorphism from $(m,c)$ to $(0,G(c))$, by the Portmanteau theorem we can deduce some in law properties of $X$ from the corresponding in law properties of $Z$. In particular, considering the process $X_\e$ with same law as $X$ but with $Y$ replaced by $Y_\e$, we can deduce its limiting behavior from the associated linear process $Z_\e$ and the regularity of $G$. Let us gather our assumptions.
\begin{assumption}\label{cor:ass}
We assume that
\begin{itemize}
\item $\frac1F$ is integrable over $(m,c)$,
\item $\alpha({\cal Y})$ is bounded from above,
\item $\EE(\sup_{\e\in(0,1]}p^\ast_\e(T))$ is finite, where $p^\ast_\e(T)$ is the counting measure at the boundary for the process $X_\e$.
\end{itemize}
\end{assumption}
The following theorem is a direct consequence of Theorem \ref{thm}.
\begin{theorem}
Under Assumption \ref{cor:ass}, the process $X_\e$ converges in law in $\mathbb{D}([0,T],\R)$ towards a process $\bar X$ such that for any measurable function $f:(m,c)\to \R$ satisfying G1) and G2) the process
\begin{align*}
&f(\bar X(t))-f(\xi_0)-\sum_{y\in\cal Y}\alpha(y)\pi(\{y\})\int_0^t f'(\bar X(s))F(\bar X(s)) \dd s\\
&\qquad-\int_0^t \int_{m}^c [f(u)-f(\bar X(s^-))]\bar\mu(\dd u) \bar p^\ast(\dd s)
\end{align*}
for $t\in[0,T]$, defined a martingale, with $\bar p^\ast$ the counting measure at the boundary for $\bar X$. The averaging measure at the boundary $\bar\mu$ is defined by
$$
\bar\mu(\dd u)=\sum_{y\in{\cal Y}}\mu_{y}(\dd u)\pi^\ast(\{1/y\})
$$
where $\pi^*$ is the invariant measure associated to the intensity matrix $\alpha(V)^{-1}Q$, thought as the generator of a $\alpha(\mathcal{Y})^{-1}$-valued continuous time Markov chain.
\end{theorem}
Here, $\alpha(V)^{-1}$ denotes the diagonal matrix such that 
$$\text{diag}(\alpha(V)^{-1})=\{1/\alpha(y)~;~y\in{\cal Y}\}.$$
Note that because of the separation of variables in the form of the flow, its value at the boundary does not appear in the expression of $\pi^\ast$, as it could be expected in more general situations.

\subsubsection{Application to a slow-fast hybrid version of a quadratic integrate-and-fire model}\label{sec:quad}

Theorem \ref{cor:ass} allows to consider other natural motions studied in the literature. For example, let us examine the following slow-fast hybrid version of a quadratic integrate-and-fire model \cite{Quad}, used in mathematical neuroscience. In such a setting, $X$ represents the membrane potential of a neural cell which is increasing until it reaches some threshold $c$, corresponding to the time where a nerve impulse is triggered, and then the potential is reset to some slower value.  For the quadratic integrate and fire model, between two jumps at the boundary $c$, the process $X$ follows the quadratic motion
$$
\frac{\dd X}{\dd t}(t)=[Y(t)X(t)]^2.
$$
In order to fix the ideas, let $Y$ be a continuous time Markov chain valued in a $\{1,2\}$ with intensity matrix given by
$$
Q=\begin{pmatrix}
-1& 1 \\
2 & -2
\end{pmatrix}
$$
such that the invariant probability measure reads
$$
\pi=\begin{pmatrix}
\frac23& \frac13 
\end{pmatrix}.
$$
For $y\in\{1,2\}$, assume that the jumping measure at the boundary $\mu_{y}$ has support $(m,c-\rho)$, for some positive constants $\rho$ and $m$ with $m+\rho<c$. The function $G$ is here given, for $x\in(m,c)$, by
$$
G(x)=\frac1m-\frac1x.
$$ 
Thus, we consider, for $t\in[0,T]$, the process
$$
Z(t)=\frac1m-\frac{1}{X(t)}.
$$
The process $Z$ is thus a piecewise linear Markov process jumping according to $Y^2$ and is constrained to the set $(0,1/m-1/c)$. The process $Y^2$ jumps at same rate as $Y$ but with state space $\{1,4\}$ instead of $\{1,2\}$. Moreover, the jump number $i$ of $Z$ is distributed according to the cumulative distribution function given, for $x\in(0,1/m-1/(c-\rho))$, by
$$
\nu_{Y_{T^{\ast -}_i}}((-\infty,x])=\PP(Z(T^\ast_1)\leq x)=\mu_{Y_{T^{\ast -}_i}}((-\infty,G^{-1}(x)]).
$$
Let us denote by $X_\e$ and $Z_\e$ the corresponding processes coupled to the process with fast dynamic $Y^2_\e$ jumping according to the intensity matrix $Q/\e$ between $1$ and $4$. The measure $\pi^\ast$ is the invariant measure on the state space $\{1/4,1\}$ associated to the intensity matrix
$$
V^{-1}Q=\begin{pmatrix} 
1& 0 \\
0 & \frac14
\end{pmatrix}
\begin{pmatrix}
-1& 1 \\
2 & -2
\end{pmatrix}=
\begin{pmatrix}
-1& 1 \\
\frac12 & -\frac12
\end{pmatrix}.
$$
That is
$$
\pi^\ast=\begin{pmatrix}
\frac13& \frac23 
\end{pmatrix}.
$$
Note that the values of $\pi$ and $\pi^\ast$ differs. According to Theorem \ref{thm}, the process $Z_\e$ converges in law in $\mathbb{D}([0,T],\R)$ towards $\bar Z$ such that for any measurable function $f : (0,1/m-1/c)\to \R$ satisfying G1) and G2) the process
\begin{align*}
&f(\bar Z(t))-f(-1/\xi_0)-\frac{4}{3}\int_0^t f'(\bar X(s)) \dd s\\
&\qquad-\int_0^t \int_{0}^{1/m-1/(c-\rho)} [f(u)-f(\bar Z(s^-))]\left[\frac13\mu_1(\dd u)+\frac23\mu_2(\dd u)\right] \bar p^\ast(\dd s)
\end{align*}
for $t\in[0,T]$, defined a martingale. As a byproduct, the process $X_\e$ converges in law in $\mathbb{D}([0,T],\R)$ towards $\bar X$ such that for any measurable function $g : (m,c)\to \R$ satisfying G1) and G2) the process
\begin{align*}
&g(\bar X(t))-g(\xi_0)-\frac{4}{3}\int_0^t g'(\bar X(s)) \bar X^2(s)\dd s\\
&\qquad-\int_0^t \int_{m}^{c-\rho} [g(u)-g(\bar X(s^-))]\left[\frac13\nu_1(\dd u)+\frac23\nu_2(\dd u)\right] \bar p^\ast(\dd s)
\end{align*}
for $t\in[0,T]$, defined a martingale.

\section{Proof of Theorem 2.1}\label{sec:proofs}

\subsection{A penalty method}

A common practice in showing tightness for constrained Markov process, is to allow the process to evolve outside of the domain for a very short time instead of having an instantaneous jump. In this line, we define a penalized process $X^P$ which is the piecewise deterministic Markov process solution of the following martingale problem. Let $k\geq1$ be an integer;  for any measurable function $f : \R\to\R$ satisfying G1) and G2), the process defined for $t\in[0,T]$ by 
\begin{align*}
&f(X^P_\e(t))-f(\xi_0)-\int_0^t f'(X^P_\e(s))Y_\e(s)\dd s\\
&\qquad-\int_0^t \int_{-\infty}^{c} [f(u)-f(X^P_\e(s))]\nu_{Y_\e(s)}(\dd u) \frac{1}{\e^k}1_{[c,\infty)}(X^P_\e(s))\dd s
\end{align*}
is a martingale. In concrete terms, the dynamic is the same as for $X_\e$ except that when beyond $c$, the process waits an exponential time of parameter $\frac{1}{\e^k}$ before jumping. The existence of such a process is inferred from its construction as a piecewise deterministic Markov process in the sense of \cite[Section 24, p. 57]{Davis93}. Due to the high intensity of jumps beyond the boundary, it is still not very comfortable to work directly on $X^P_\e$ to show its tightness. As explained in \cite[Section 6.4, p. 165]{Kurtz90}, we can slow down the process beyond the boundary to overcome this difficulty. For this purpose, we define the following random time-change, for $t\in[0,T]$,
$$
\lambda_\e(t)=\int_0^t \frac{\dd s}{1+\frac{1}{\e^k}1_{[c,\infty)}(X^P_\e(s))},\quad \mu_\e(t)=t-\lambda_\e(t).
$$
Being continuous and strictly increasing, the process $\lambda_\e$ defines a well defined time-change. Notice that $\mu_\e+\lambda_\e=\text{Id}$ and since $\mu_\e$ and $\lambda_\e$ are increasing, for any $0\leq t\leq t+h\leq T$,
$$
\lambda_\e(t+h)-\lambda_\e(t)\leq h\quad\text{and}\quad\mu_\e(t+h)-\mu_\e(t)\leq h.
$$
The increases of $\lambda_\e$ and $\mu_\e$ are thus bounded by the increases of the identity, uniformly in $\e$. The following lemma characterized the limit behavior of $\lambda_\e$.
\begin{lemma}\label{lem:lambda}
We have, in law,
$$
\lim_{\e\to0}\|\mu_\e\|_\infty=\lim_{\e\to0}\|\lambda_\e-{\rm Id}\|_{\infty}=0.
$$
\end{lemma}
\begin{proof}
For any $t\in[0,T]$,
$$
t-\lambda_\e(t)=\frac{1}{1+\e^k}\int_0^t 1_{[c,\infty)}(X^P_\e(s))\dd s\leq\frac{1}{1+\e^k}\int_0^T 1_{[c,\infty)}(X^P_\e(s))\dd s.
$$
Since $\sup_{\e\in(0,1]}p^\ast_\e(T)$ is finite $\PP$-almost-surely, the time spent beyond $c$ for $X^P_\e$ has same law as the sum of a finite number of exponential variable of parameters of order $1/\e^k$ almost-surely, yielding the result.
\end{proof}
This does not mean that $\mu_\e(\dd t)$ goes to zero. Intuitively, it should rather converge towards $p^*(\dd t)$.

Now, we define the time-change processes, for $t\in[0,\lambda^{-1}_\e(T)]$, by
$$
U_\e(t)=X^P_\e\circ \lambda_\e(t)\quad\text{and}\quad V_\e(t)=Y_\e\circ \lambda_\e(t).
$$
Notice that $\lambda^{-1}_\e(T)\geq T$. Then, for any measurable function $f : \R\to\R$ satisfying G1) and G2), the process defined by
\begin{align*}
&f(U_\e(\cdot))-f(U_\e(0))-\int_0^\cdot f'(U_\e(s))V_\e(s)\lambda_\e(\dd s)\\
&\qquad-\int_0^\cdot \int_{-\infty}^c [f(u)-f(U_\e(s))]\nu_{V_\e(s)}(\dd u)\mu_\e(\dd s)
\end{align*}
is a martingale. The process is now in an enough standard form to apply classical tightness theorems of the literature.
\begin{proposition}
For any time horizon $\tilde T$, the process $(U_\e,\lambda_\e)$ is tight for the Skorokhod topology on real càdlàg functions on $[0,\tilde T]$.
\end{proposition}
\begin{proof}
Since the increases of $\lambda_\e$ and $\mu_\e$ are dominated, uniformly in $\e$, by the increases of the identity, and since $\cal Y$ is bounded, this is a direct application of  \cite[Theorem 9.4, p. 145]{EK}. 
\end{proof}
Lemma \ref{lem:lambda} identify the limit of $\lambda_\e$ as being the identity, thus a strictly increasing function. From this fact we deduce from \cite[Theorem 1.1]{Kurtz91} the tightness of the penalized process as stated below.
\begin{proposition}
For any time horizon $T$, the family $\{X^P_\e,\e\in(0,1]\}$ is tight for the Skorokhod topology on real càdlàg functions on $[0,T]$.
\end{proposition}

\subsection{Coupling and tightness for the initial process}

The aim of this part is to show that the family $\{X_\e,\e\in(0,1]\}$ is tight. To this end, we show that $X_\e$ and $X^P_\e$ are close enough such that the existence of a converging subsequence for the first one infers the existence of such a subsequence for the second one. Let us describe the coupling procedure, in emphasizing the role of $k$ by denoting $X^P_\e$ by $X^{P_k}_\e$.
\medbreak
{\bf Coupling procedure:}
\begin{itemize}
\item {\bf Glue the two processes until the first hitting time of the boundary:} the two starting points are the same: $X^\e(0)=X^{P_k}_\e(0)=\xi_0$. Then, for $t\in[0,T^{*}_{\e,1})$, $X^\e(t)=X^{P_k}_\e(t)$. Let us denote by $T^{*,P_k}_{\e,1}$ the first jumping time for $X^{P_k}_\e$.
\item {\bf The two processes jump to the same place:} As the jumping measure depends only on $Y_\e$, we can set $X_\e(T^{*}_{\e,1})=X^{P_k}_\e(T^{*,P_k}_{\e,1})$.
\item {\bf Let the two processes evolve but with same post-jump-value location:} always set $X_\e(T^{*}_{\e,i})=X^{P_k}_\e(T^{*,P_k}_{\e,i})$ for $1\leq i\leq p^\ast_\e(T)$.
\end{itemize}
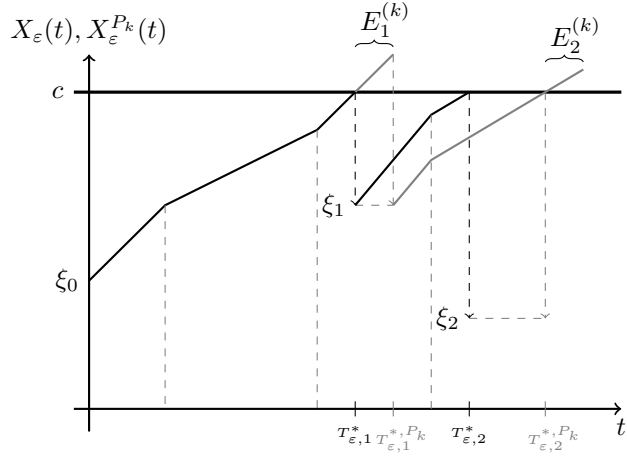
\begin{figure}
\begin{center}
\begin{tikzpicture}
\draw[thick,->](0,0)--(0,5) node[above]{$X_\e(t), X^{P_k}_\e(t)$};
\draw[very thick](-0.2,4.5) node[left]{$c$}--(7,4.5);
\draw[thick] (0,2) node[left]{$\xi_0$}--(1,3)--(3,4)--(3.5,4.5);
\draw[thick,gray] (3.5,4.5)--(4,5);
\draw[dashed,gray] (3.5,3)--(4,3);
\draw [decorate,decoration=brace] (3.5,5.1)--(4,5.1);
\draw (3.9,5.1) node[above]{$E^{(k)}_1$};
\draw[dashed,->] (3.5,4.5)--(3.5,3) node[left]{$\xi_1$};
\draw[dashed,->,gray] (4,5)--(4,3);
\draw[thick] (3.5,3)--(4.5,4.2)--(5,4.5);
\draw[thick,gray] (4,3)--(4.5,3.6)--(6,4.5)--(6.5,4.8);
\draw[dashed,->] (5,4.5)--(5,1.5) node[left]{$\xi_2$};
\draw[dashed,->,gray] (6,4.5)--(6,1.5);
\draw[dashed,gray] (5,1.5)--(6,1.5);
\draw [decorate,decoration=brace] (6,4.9)--(6.5,4.9);
\draw (6.4,4.9) node[above]{$E^{(k)}_2$};
\draw[thick,->](-0.2,0.3)--(7,0.3) node[below]{$t$}; 
\draw[dashed,gray] (1,3)--(1,0.3);
\draw[dashed,gray] (3,4)--(3,0.3);
\draw[dashed,gray] (4.5,4.2)--(4.5,0.3);
\draw (3.5,0.4)--(3.5,0.2) node[below]{\tiny$T^\ast_{\e,1}$};
\draw (5,0.4)--(5,0.2) node[below]{\tiny$T^\ast_{\e,2}$};
\draw[gray] (4,0.4)--(4,0.2) node[below]{\tiny$\quad T^{\ast,P_k}_{\e,1}$};
\draw[gray] (6,0.4)--(6,0.2) node[below]{\tiny$\quad T^{\ast,P_k}_{\e,2}$};
\end{tikzpicture}
\caption{The coupling between $X_\e$ (in black) and $X^{P_k}_\e$ (in gray). The random variables $E^{(k)}_1$ and $E^{(k)}_1$ have same law as independent and exponentially distributed random variables with parameter $1/\e^k$: they represent the time spent beyond $c$ for the process $X^{P_k}_\e$.}\label{fig:coupling}
\end{center}
\end{figure}
This coupling, illustrated in Figure \ref{fig:coupling}, has good properties, such as the fact that $X^{P_k}_\e$ always jumps after $X_\e$. Another one is emphasized in the following proposition.
\begin{lemma}\label{lem:Yk}
The probability that $Y_\e$ jumps between $T^{*}_{\e,1}$ and $T^{*,P_k}_{\e,1}$ goes to $0$ when $k$ goes to infinity.
\end{lemma}
\begin{proof}
Remark that $T^{*,P_k}_{\e,1}-T^{*}_{\e,1}$ is dominated by an exponential variable of parameter $\frac{1}{\e^k}$ which is moreover independent of $Y_\e$. Therefore, the probability that $Y_\e$ jumps between $T^{*}_{\e,1}$ and $T^{*,P_k}_{\e,1}$ is asymptotically (in $k$) dominated by the probability that $Y_\e$ jumps exactly at time $T^{*}_{\e,1}$, which is zero.
\end{proof}
This implies that for $k$ big enough, with high probability (how high depending on $k$),  after its first jump, $X^{P_k}_\e$ has same direction as $X_\e$ and, $\mathcal{Y}$ being bounded, their distance is of order an exponential variable of order $\frac{1}{\e^k}$, see Figure \ref{fig:coupling}. Then, the distance between the two processes remains the same until $X_\e$ reaches again the boundary $\{c\}$.

Let us recall that the Wasserstein distance between $X_\e$ and $X^{P_k}_\e$ is defined as
$$
W(X_\e,X^{P_k}_\e)=\inf_{\substack{A\sim X_\e\\ ~~B\sim X^{P_k}_\e}}\EE({\dd}_S(A,B)),
$$
where $\dd_S$ is the Skorokhod distance, defined, for $\Lambda$ the set of continuous one-to-one mapping of $[0,T]$, by
$$
{\dd}_S(A,B)=\inf_{\lambda\in\Lambda}\max(\|\lambda-{\rm Id}\|_\infty,\|A-B\circ\lambda\|_\infty).
$$
Thus, for the Wasserstein distance to go to zero, it is enough to find a coupling of $X_\e$ and $X^{P_k}_\e$ such that their Skorokhod distance goes to zero in expectation. In view of Lemma \ref{lem:Yk} and the fact that the number of jumps of $X_\e$ is bounded almost-surely, this is what is achieve by our coupling procedure.
\begin{proposition}\label{prop:W}
For any $\e_0 \in(0,1)$,
$$
\lim_{k\to\infty}\sup_{\e\in(0,\e_0)}W(X_\e,X^{P_k}_\e)=0.
$$ 
\end{proposition}
\begin{proof}
Let us use the flexibility of the Wasserstein and Skorokhod distances in considering the defined coupling together with the one-to-one mapping
$$
\gamma^{(k)}_\e(t)=\frac{T^{\ast,P_k}_{p^{\ast,P_k}_\e(t)+1,\e}+E^{(k)}_{p^{\ast,P_k}_\e(t)+1}}{T^\ast_{p^\ast_\e(t)+1,\e}}t,
$$
defined for $t\in[0,T]$, where $(E^{(k)}_i)_{i\geq1}$ are the succesive times spent beyond $c$ for the process $X^{P_k}_\e$ and thus have same law as independent exponential random variables with parameter $\frac{1}{\e^k}$. This can be seen as a homothety with random piecewise constant ratio. Remark that the map $\gamma^{(k)}_\e$ is defined such that the two processes $X_\e$ and $X^{P_k}_\e\circ\gamma^{(k)}_\e$ jumps at the same time and are glued to the same value after jumps, as illustrated in Figure \ref{fig:proof}. Quite exactly as for Lemma \ref{lem:lambda}, we can show that the piecewise constant ratio of the homothety $\gamma^{(k)}_\e$ goes to one when $k$ goes to infinity, uniformly in $t\in[0,T]$ and $\e\in(0,\e_0)$. This implies that, in expectation, the uniform distance between $X_\e$ and $X^{P_k}_\e\circ\gamma^{(k)}_\e$ goes to zero when $k$ goes to infinity, uniformly in $\e\in(0,\e_0)$. 
\end{proof}
\begin{figure}
\begin{center}
\def\tr{0}
\def\dr{-6}
\begin{tikzpicture}[scale=0.8,every node/.style={scale=0.8}]
\draw[thick,->](0,0)--(0,5);
\draw[very thick](-0.2,4.5) node[left]{$c$}--(7,4.5);
\draw[thick] (0,2) node[left]{$\xi_0$}--(1,3)--(3,4)--(3.5,4.5);
\draw[thick,gray] (3.5,4.5)--(4,5);
\draw[dashed,gray] (3.5,3)--(4,3);
\draw [decorate,decoration=brace] (3.5,5.1)--(4,5.1);
\draw (3.9,5.1) node[above]{$E^{(k)}_1$};
\draw[dashed,->] (3.5,4.5)--(3.5,3) node[left]{$\xi_1$};
\draw[dashed,->,gray] (4,5)--(4,3);
\draw[thick] (3.5,3)--(4.5,4.2)--(5,4.5);
\draw[thick,gray] (4,3)--(4.5,3.6)--(6,4.5)--(6.5,4.8);
\draw[dashed,->] (5,4.5)--(5,1.5) node[left]{$\xi_2$};
\draw[dashed,->,gray] (6,4.5)--(6,1.5);
\draw[dashed,gray] (5,1.5)--(6,1.5);
\draw [decorate,decoration=brace] (6,4.9)--(6.5,4.9);
\draw (6.4,4.9) node[above]{$E^{(k)}_2$};
\draw[thick,->](-0.2,0.3)--(7,0.3) node[below]{$t$}; 
\draw[dashed,gray] (1,3)--(1,0.3);
\draw[dashed,gray] (3,4)--(3,0.3);
\draw[dashed,gray] (4.5,4.2)--(4.5,0.3);
\draw (3.5,0.4)--(3.5,0.2) node[below]{\tiny$T^\ast_{\e,1}$};
\draw (5,0.4)--(5,0.2) node[below]{\tiny$T^\ast_{\e,2}$};
\draw[gray] (4,0.4)--(4,0.2) node[below]{\tiny$\quad T^{\ast,P_k}_{\e,1}$};
\draw[gray] (6,0.4)--(6,0.2) node[below]{\tiny$\quad T^{\ast,P_k}_{\e,2}$};
\draw[thick,->](0+\tr,0+\dr)--(0+\tr,5+\dr);
\draw[very thick](-0.2+\tr,4.5+\dr) node[left]{$c$}--(7+\tr,4.5+\dr);
\draw[thick] (0+\tr,2+\dr) node[left]{$\xi_0$}--(1+\tr,3+\dr)--(3+\tr,4+\dr)--(3.5+\tr,4.5+\dr);
\draw[dashed,->] (3.5+\tr,4.5+\dr)--(3.5+\tr,3+\dr) node[left]{$\xi_1$};
\draw[thick] (3.5+\tr,3+\dr)--(4.5+\tr,4.2+\dr)--(5+\tr,4.5+\dr);
\draw[dashed,->] (5+\tr,4.5+\dr)--(5+\tr,1.5+\dr) node[left]{$\xi_2$};
\draw[thick,gray] (0+\tr,2+\dr)--(0.5+\tr,3-0.2+\dr)--(2.5+\tr,4+\dr)--(3.5+\tr,5+\dr);
\draw[dashed,->,gray] (3.5+\tr,5+\dr)--(3.5+\tr,3+\dr);
\draw[thick,gray] (3.5+\tr,3+\dr)--(4+\tr,4+\dr)--(5+\tr,5+\dr);
\draw[dashed,->,gray] (5+\tr,5+\dr)--(5+\tr,1.5+\dr);
\draw[thick,->](-0.2+\tr,0.3+\dr)--(7+\tr,0.3+\dr) node[below]{$t$}; 
\draw (3.5+\tr,0.4+\dr)--(3.5+\tr,0.2+\dr) node[below]{\tiny$T^\ast_{\e,1}$};
\draw (3.5+\tr,-0.2+\dr) node[below]{\tiny$=$};
\draw (3.5+\tr,-0.4+\dr) node[below]{\tiny$\gamma^{(k)}_\e(T^{\ast,P_k}_{\e,1})$};
\draw (5+\tr,0.4+\dr)--(5+\tr,0.2+\dr)node[below]{\tiny$T^\ast_{\e,2}$};
\draw (5+\tr,-0.2+\dr) node[below]{\tiny$=$};
\draw (5+\tr,-0.4+\dr) node[below]{\tiny$\gamma^{(k)}_\e(T^{\ast,P_k}_{\e,2})$};
\draw [decorate,decoration=brace] (3.55+\tr,5+\dr) -- node[right]{\tiny O$(E^{(k)}_1)$}(3.55+\tr,4.5+\dr) ;
\draw [decorate,decoration=brace] (5.05+\tr,5+\dr) -- node[right]{\tiny O$(E^{(k)}_2)$}(5.05+\tr,4.5+\dr) ;
\draw[thick,dashed,->](7.25,0.5) to[bend left]  (7.25,-1.75);
\draw (8,-0.25) node[below]{$\gamma^{(k)}_\e$};
\end{tikzpicture}
\caption{Illustration of the proof of Proposition \ref{prop:W}. Top: trajectories of $X_\e$ (in black) and $X^{P_k}_\e$ (in gray). Bottom:  trajectories of $X_\e$ (in black) and $X^{P_k}_\e\circ\gamma^{(k)}_\e$ (in gray).}\label{fig:proof}
\end{center}
\end{figure}
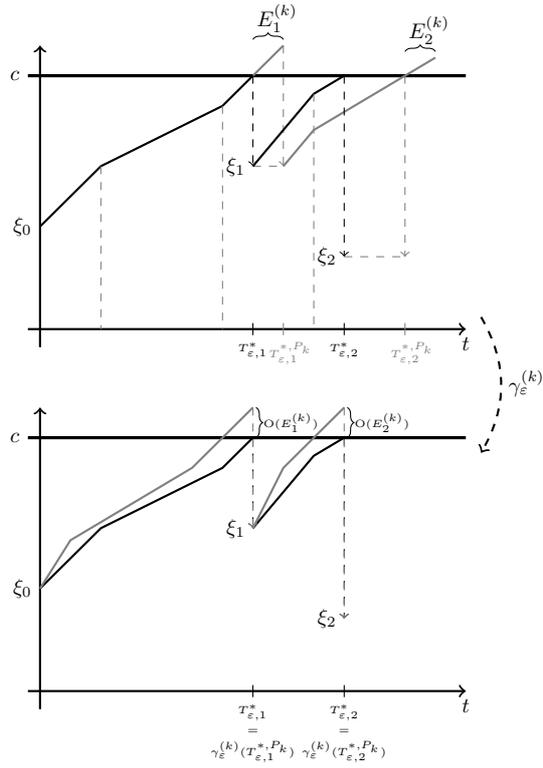
Since convergence in Wasserstein distance implies in law convergence, we can state the following proposition.
\begin{proposition}
For any time horizon $T$, the family $\{X_\e,\e\in(0,1]\}$ is tight for the Skorokhod topology on real càdlàg functions on $[0,T]$.
\end{proposition}
\begin{proof}
Let us write $\dd_{\cal L}$ for a distance metrizing convergence in law and let $\e_0\in(0,1)$ be fixed. According to Proposition \ref{prop:W}, for any $\eta>0$ we can find some $k$ such that for any $\e\in(0,\e_0)$,
$$
\dd_{\cal L}({\cal L}(X_\e),{\cal L}(X^{P_k}_\e))\leq \frac{\eta}{2}.
$$
Then, writing $\bar X$ for an accumulation point of the family $\{X^{P_k}_\e,\e\in(0,\e_0)\}$, there is some $\e\in(0,\e_0)$ such that
$$
\dd_{\cal L}({\cal L}(X^{P_k}_\e),{\cal L}(\bar X))\leq \frac{\eta}{2}.
$$
Hence the result.
\end{proof}

\subsection{Finite dimensional laws for the limit}

Let us denote by $\Lambda_\e$ the occupation measure on $[0,T]\times \mathcal{P}({\cal Y})$ defined by
$$
\Lambda_\e([0,t]\times \{y\})=\int_0^t1_{y}(Y_\e(s))\dd s.
$$
According to the ergodic theorem (recall that $Y_\e=Y(\cdot/\e)$ with $Y$ a positive recurrent continuous time Markov chain), this measure converges in law when $\e$ goes to zero to the measure $\bar\Lambda$ defined on $[0,T]\times \mathcal{P}({\cal Y})$ by
$$
\bar\Lambda([0,t]\times \{y\})=t\sum_{y\in\cal Y}y\pi(\{y\}).
$$
In the following, we denote by $(\bar X,\bar \Lambda)$ an accumulation point of the family\\$\{X_\e,\Lambda_\e ;~ \e\in(0,1)\}$. Let us recall that for any measurable function $f$ satisfying G1) and G2) the process defined, for $t\in[0,T]$, by
\begin{align*}
&f(X_\e(t))-f(\xi_0)-\int_0^t f'(X_\e(s))Y_\e(s)\dd s\\
&-\int_0^t \int_{-\infty}^c [f(u)-f(X_\e(s^-))]\nu_{Y_\e(s^-)}(\dd u) p^\ast_\e(\dd s)
\end{align*}
is a martingale. In light of  \cite[Theorem 2.1]{Kurtz92}, this is not hard to see, as in the context of averaging without constraints, that the term
$$
\int_0^t f'(X_\e(s))Y_\e(s)\dd s=\int_{[0,t]\times\cal Y} f'(X_\e(s))y\Lambda_\e(\dd s,\dd y)
$$
converges in law towards
\begin{equation}\label{eq1}
\int_{[0,t]\times\cal Y} f'(\bar X(s))y\Lambda(\dd s,\dd y)=\sum_{y\in\cal Y}y \pi(\{y\})\int_0^t f'(\bar X(s))\dd s.
\end{equation}
The integral with respect to the singular measure $p^\ast_\e$ requires a specific attention. We expand it as
\begin{align*}
&\int_0^t \int_{-\infty}^c [f(u)-f(X_\e(s^-))]\nu_{Y_\e(s^-)}(\dd u) p^\ast_\e(\dd s)\\
=&\sum_{y\in\cal Y}\sum_{i\geq1}\int_{-\infty}^c [f(u)-f(c)]\nu_{y}(\dd u)1_{Y_\e(T^{\ast,-}_{i,\e})=y~;~i\leq{p^\ast_\e(t)}}.
\end{align*}
Notice that $X_\e$ is strictly increasing in between two jumps, thus invertible in such a time window. The reciprocal process is defined until the first jumps of $X^\e$ as
$$
\frac{1}{\xi_0}+\int_0^t W_\e(s)\dd s
$$
where $W_\e=W(\cdot/\e)$ with $W$ a $\left\{1/y~;~y\in{\cal Y} \right\}$-valued continuous time Markov chain with intensity matrix $V^{-1}Q$. We thus consider a ``mirror" process $M_\e$, as illustrated in Figure \ref{fig:proof:id}, starting at time $0$ and evolving according to a continuous piecewise linear motion with speed given by $W_\e$:
$$
\forall x\geq0,\quad M_\e(x)=\int_0^x W_\e(u)\dd u.
$$
Recall that we write $\pi^\ast$ for the invariant measure associated to $W$.
\begin{lemma}\label{lem:id}
Let us denote by $\nu_{\xi_0}$ the law of the initial condition $\xi_0$ and $E_{\pi^\ast}$ the first moment of $\pi^\ast$. The sequence $(T^\ast_{\e,i},Y_\e(T^{\ast,-}_{\e,i}))_{1\leq i\leq p^\ast_\e(T)}$ converges in law when $\e$ goes to zero towards $(\bar T^\ast_{i},Z_i)_{1\leq i\leq p^\ast(T)}$, with law given, for any $k\geq1$,  any sequence of times $t_1,\ldots, t_k$ and any sequence of values $x_0,\cdots,x_k$, by:
\begin{align*}
&\PP\left(\bigcap_{i=1}^k \{\bar T^\ast_{i}\leq t_i\} \cap \{Z_i=x_i\}\cap \{p^\ast(T)=k\}\right)\\
=&\int_{(-\infty,c)^{k+1}}\nu_{\xi_0}(\dd u_0)\ldots \nu_{x_{k}}(\dd u_{k})\Pi_{i=1}^k \pi^\ast\left(\left\{1/x_i\right\}\right)\\
&\quad\times1_{\{\left(ic-\sum_{j=0}^{i-1}u_j\right)E_{\pi^\ast}\leq t_i;\left(kc-\sum_{j=0}^{k-1}u_j\right)E_{\pi^\ast}\leq T<\left((k+1)c-\sum_{j=0}^{k}u_j\right)E_{\pi^\ast}\}}.
\end{align*}
\end{lemma}
\begin{proof}
We consider at first the case $k=1$. As illustrated in Figure \ref{fig:proof:id} we have
\begin{align*}
&\PP\left(\{T^\ast_{\e,1}\leq t_1\} \cap \{Y_\e(T^{\ast,-}_{\e,1})=x_1\}\cap \{p^\ast_\e(T)= 1\}\right)\\
=&\PP\bigg(\{\int_0^{c-\xi_0}W(s/\e)\dd s\leq t_1\} \cap \{W((c-\xi_0)/\e)=\frac{1}{x_1}\}\cap \{\int_0^{c-\xi_0}W(s/\e)\dd s\leq T\}\\
&\qquad\cap\{\int_0^{2c-(\xi_0+\xi_1)}W(s/\e)\dd s>T\}\bigg),
\end{align*}
where $\xi_0$ and $\xi_1$ are independents, with laws $\nu_{\xi_0}$ and $\nu_{x_1}$ on the event $\{Y_\e(T^{\ast,-}_{\e,1})=x_1\}$. Thus,
\begin{align*}
&\PP\left(\{T^\ast_{\e,1}\leq t_1\} \cap \{Y_\e(T^{\ast,-}_{\e,1})=x_1\}\cap \{p^\ast_\e(T)= 1\}\right)\\
=&\int_{(-\infty,c)^2}\nu_{\xi_0}(\dd u_0)\nu_{x_1}(\dd u_1)\PP\bigg(\{\int_0^{c-u_0}W(s/\e)\dd s\leq t_1\}\cap \{W((c-u_0)/\e)=\frac{1}{x_1}\}\\
&\qquad  \cap \{\int_0^{c-u_0}W(s/\e)\dd s\leq T\}\cap\{\int_0^{2c-(u_0+u_1)}W(s/\e)\dd s>T\}\bigg).
\end{align*}
By the ergodic theorem and dominated convergence, this latter term goes to
$$
\int_{(-\infty,c)^2}\nu_{\xi_0}(\dd u_0)\nu_{x_1}(\dd u_1)\pi^\ast\left(\{1/x_1\}\right)1_{(c-u_0)E_{\pi^\ast}\leq T<(2c-(u_0+u_1))E_{\pi^\ast}}
$$
as required. In the same line, for any $k\geq1$, for any sequence of times $t_1,\ldots, t_k$ and any sequence of values $x_0,\cdots,x_k$, considering all possible post-jump value locations, we have,
\begin{align*}
&\PP\left(\bigcap_{i=1}^k \{T^\ast_{\e,i}\leq t_i\} \cap \{Y_\e(T^{\ast,-}_{\e,i})=x_i\}\cap \{p^\ast_\e(T)= k\}\right)\\
&\PP\left(\bigcap_{i=1}^k \{T^\ast_{\e,i}\leq t_i\} \cap \{Y_\e(T^{\ast,-}_{\e,i})=x_i\}\cap \{T^\ast_{\e,k}\leq T\}\cap\{T^\ast_{\e,k+1}> T\}\right)\\
=&~\int_{(-\infty,c)^{k+1}}\nu_{\xi_0}(\dd u_0)\ldots \nu_{x_{k}}(\dd u_{k})\\
&\quad\times\PP\bigg(\bigcap_{i=1}^k \{\int_{0}^{ic-\sum_{j=0}^{i-1}u_j}W(s/\e)\dd s\leq t_i\} \cap \{W((ic-\sum_{j=0}^{i-1}u_j)/\e)=\frac{1}{x_i}\}\\
&\qquad\cap \{\int_{0}^{kc-\sum_{j=0}^{k-1}u_j}W(s/\e)\dd s\leq T\}\cap \{\int_{0}^{(k+1)c-\sum_{j=0}^{k}u_j}W(s/\e)\dd s\leq T\}\bigg).
\end{align*}
By the ergodic theorem and dominated convergence, this latter term goes to
\begin{align*}
&\int_{(-\infty,c)^{k+1}}\nu_{x_0}(\dd u_0)\ldots \nu_{x_{k}}(\dd u_{k})\Pi_{i=1}^k \pi^\ast\left(\left\{1/x_i\right\}\right)\\
&\qquad\times1_{\{\left(ic-\sum_{j=0}^{i-1}u_j\right)E_{\pi^\ast}\leq t_i;\left(kc-\sum_{j=0}^{k-1}u_j\right)E_{\pi^\ast}\leq T;\left(kc-\sum_{j=0}^{k-1}u_j\right)E_{\pi^\ast}> T\}}
\end{align*}
when $\e$ goes to zero, as required.
\end{proof}

\begin{figure}
\begin{center}
\def\tr{1.5}
\def\dr{-8.5}
\begin{tikzpicture}[scale=0.8,every node/.style={scale=0.8}]
\draw[thick,->](0,0)--(0,5) node[left]{$X_\e(t)$};
\draw[very thick](-0.2,4.5) node[left]{$c$}--(7,4.5);
\draw[thick] (0,2) node[left]{$\xi_0$}--(1,3)--(3,4)--(3.5,4.5);
\draw[dashed,->] (3.5,4.5)--(3.5,3) node[left]{$\xi_1$};
\draw[thick,gray] (3.5,3)--(4.5,4.2)--(5,4.5);
\draw[dashed,->] (5,4.5)--(5,1.5) node[left]{$\xi_2$} node[right]{$etc...$};
\draw[thick,->](-0.2,0.3)--(7,0.3) node[below]{$t$}; 
\draw[dashed,gray] (1,3)--(1,0.3);
\draw[dashed,gray] (3,4)--(3,0.3);
\draw[dashed,gray] (4.5,4.2)--(4.5,0.3);
\draw (3.5,0.4)--(3.5,0.2) node[below]{\tiny$T^\ast_{\e,1}$};
\draw (5,0.4)--(5,0.2) node[below]{\tiny$T^\ast_{\e,2}$};

\draw[->](0+\tr,1.2+\dr) --(0+\tr,7+\dr) node[above left]{$M_\e(x)$};
\draw[->](-0.2+\tr,1.5+\dr)node[below]{$0$} --(5+\tr,1.5+\dr) node[below]{$x$};
\draw[thick](2.5+\tr,1.3+\dr)node[below]{$c-\xi_0$} --(2.5+\tr,1.6+\dr);
\draw[thick](4+\tr,1.4+\dr) --(4+\tr,1.6+\dr);
\draw[thick](4.5+\tr,1.3+\dr)node[below]{$2c-(\xi_0+\xi_1)$};
\draw[thick] (0+\tr,1.5+\dr)--(1+\tr,2.5+\dr)--(2+\tr,4.5+\dr)--(2.5+\tr,5+\dr);
\draw[thick,gray] (2.5+\tr,5+\dr)--(3.5+\tr,5.83+\dr)--(4+\tr,6.66+\dr)node[right]{$etc...$};
\draw (4+\tr,6.66+\dr)node[right]{$etc...$};
\draw[very thick] (2.5+\tr,1.4+\dr)--(2.5+\tr,7+\dr);
\draw[very thick] (4+\tr,1.4+\dr)--(4+\tr,7+\dr);
\draw[dashed,gray] (0+\tr,2.5+\dr)--(1+\tr,2.5+\dr);
\draw[dashed,gray] (0+\tr,4.5+\dr)--(2+\tr,4.5+\dr);
\draw[dashed,gray] (0+\tr,5.83+\dr)--(3.5+\tr,5.83+\dr);
\draw (-0.1+\tr,5+\dr)node[left]{\tiny$T^\ast_{\e,1}$}--(0.1+\tr,5+\dr);
\draw (-0.1+\tr,6.66+\dr)node[left]{\tiny$T^\ast_{\e,2}$}--(0.1+\tr,6.66+\dr);
\draw[thick,dashed,->](7.25,0.5) to[bend left] node[ right]{mirroring} (7.25,-2);
\end{tikzpicture}
\caption{Illustration of the proof of Lemma \ref{lem:id}. Top: a trajectory of $X_\e$. Bottom: a trajectory of the ``mirror'' process $M_\e$. The process $X_\e$ hits the boundary at time corresponding to $c-\xi_0$, $2c-(\xi_0+\xi_1)$... for the mirror process $M_\e$. This shows that $X_\e$ hits the boundary for the first time with a given speed $y$ if and only if the mirror process evolves at speed $1/y$ at time $c-\xi_0$, and so on and so forth for the other hitting times of the boundary.}\label{fig:proof:id}
\end{center}
\end{figure}
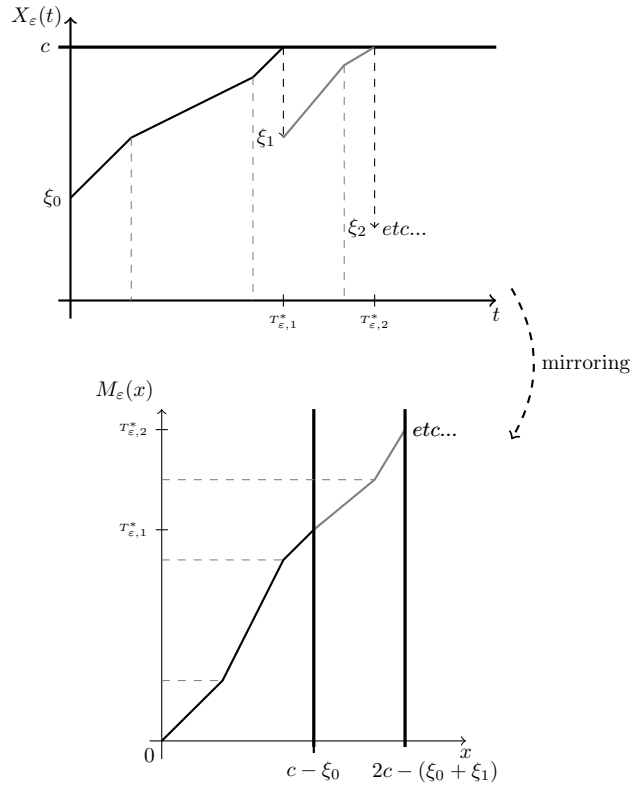

This is then routine (see the proof of \cite[Theorem 2.1]{Kurtz92}) to show that Lemma \ref{lem:id} and Equation \ref{eq1} implies that if the function $f:~(-\infty,c)\to \R$ is such that
\begin{itemize}
\item[G1)] $f$ is measurable and absolutely continuous with respect to the Lebesgue measure;
\item[G2)] $f$ is locally integrable at the boundary: for any $t\in[0,T]$,
$$
\EE\left(\sum_{T^\ast_i\leq t}|f(X(\bar T^\ast_i))-f(X(\bar T^{\ast-}_i))|\right)<\infty.
$$
\end{itemize}
then the process 
\begin{align*}
&f(\bar X(t))-f(\xi_0)-E_{\pi^\ast}\int_0^t f'(\bar X(s))\dd s\\
&-\int_0^t \sum_{y\in\cal Y}\int_{-\infty}^c [f(u)-f( \bar X(s^-)]\nu_y(\dd u)\pi^\ast\left(\left\{1/y\right\}\right)p^\ast(\dd s)
\end{align*}
is a martingale, which is precisely Theorem \ref{thm}.

\textbf{Acknowledgements.} The author is thankful to Professor François Dufour for motivating and enlightening discussions.

\end{document}